\documentclass{scrartcl}
\KOMAoptions{parskip=half,paper=a4,abstract=on}

\usepackage{amsmath, amsthm, amssymb, bbm, mathtools, nicefrac}
\usepackage[main=english,brazil]{babel}
\usepackage[utf8]{inputenc}
\usepackage[pdftex,dvipsnames]{xcolor}  
\usepackage{hyperref}
\usepackage{url}
\usepackage{fixme}
\usepackage{multicol}
\usepackage{xspace}
\usepackage[all]{xy}
\usepackage{enumerate}

\newcommand*{\A}{\mathcal{A}}
\newcommand*{\C}{\mathcal{C}}

\newcommand*{\E}{\mathcal{E}}
\newcommand*{\G}{\mathcal{G}}

\newcommand*{\N}{\mathcal{N}}

\newcommand*{\R}{\mathcal{R}}

\newcommand{\GD}{\mathcal{G}^{(2)}}
\newcommand{\ED}{\E^{(2)}}
\newcommand{\GZ}{\mathcal{G}^{(0)}}

\newcommand{\HH}{\mathcal{H}}
\newcommand{\HZ}{\mathcal{H}^{(0)}}

\newcommand{\id}{\operatorname{id}} 
\newcommand{\iso}{\operatorname{Iso}} 

\DeclareMathOperator{\Ext}{Ext}

\newcommand{\cf}{\mbox{cf.}\xspace}				
\newcommand{\eg}{\mbox{e.\,g.}\xspace}			
\newcommand*{\ie}{\mbox{i.\,e.}\xspace}			
\newcommand{\wilog}{\mbox{w.\,l.\,o.\,g.}\xspace}	
\newcommand{\wrt}{\mbox{w.\,r.\,t.}\xspace}		
\newcommand*{\ndash}{\nobreakdash-}
\newcommand*{\Star}{$^*$\ndash}

\theoremstyle{plain}
	\newtheorem{teo}{Theorem}[section]
	\newtheorem{lemma}[teo]{Lemma}
	\newtheorem{cor}[teo]{Corollary}
	\newtheorem{prop}[teo]{Proposition}
	
\theoremstyle{definition}
	\newtheorem{defi}[teo]{Definition}
	\newtheorem{rem}[teo]{Remark}
	\newtheorem{example}[teo]{Example}

\hypersetup{colorlinks=true,
  linkcolor = black,
  citecolor = blue,
  urlcolor =  blue, 
}

\title{Non-Abelian extensions of groupoids\\ and their groupoid rings}
\date{\today}

\begin{document}

\author{Natã Machado \thanks{
	    Universidade Federal de Santa Catarina,
		\href{mailto:nmachado92@gmail.com}{\nolinkurl{nmachado92@gmail.com}}
	} \and
Johan \"Oinert \thanks{
		Blekinge Tekniska H\"ogskola,
		\href{mailto:johan.oinert@bth.se}{\nolinkurl{johan.oinert@bth.se}}
	} \and 
Stefan Wagner \thanks{
		Blekinge Tekniska H\"ogskola,
		\href{mailto:stefan.wagner@bth.se}{\nolinkurl{stefan.wagner@bth.se}}
	}
}
\sloppy
\maketitle

\begin{abstract}
\noindent
    We present a geometrically oriented classification theory for \mbox{non-Abelian} extensions of groupoids generalizing the classification theory for Abelian extensions of groupoids by Westman as well as the familiar classification theory for non-Abelian extensions of groups by Schreier and Eilenberg-Mac~Lane.
    As an application of our techniques we~demonstrate that each extension of groupoids $\N \to \E \to \G$ gives rise to a groupoid crossed~product of $\G$ by the groupoid ring of $\N$ which recovers the groupoid ring of~$\E$ up to isomorphism.
    Furthermore, we make the somewhat surprising observation that our classification methods naturally transfer to the class of groupoid crossed products, thus providing a classification theory for this class of rings.
    Our study is motivated by the search for natural examples of groupoid crossed products.
    
\vspace*{0.1cm}

\vspace*{0.1cm}
\end{abstract}

\listoffixmes


\section{Introduction}

The problem of classifying all extensions of a given group $G$ by a group $N$ is a core problem in group theory and may be found in many expositions.
The first systematic treatments seem to originate in Schreier’s PhD thesis from 1923 (see also \cite{Schreier1925}) and in the work of Baer~\cite{Baer34} from the 1930s.
Cohomological methods used to study group extensions first appeared in the seminal papers by Eilenberg and MacLane~\cite{MacEil42,EilenbergMacLane}.
Another curious reference is due to computer scientist Alan Turing~\cite{Turing38}.
The central concept underlying a group extension is that of a so-called \emph{factor system},
which determines and is determined by the group extension.
Examples and applications of group extensions can be found in almost all disciplines of modern mathematics.
For instance, non-Abelian extensions of Lie groups occur quite naturally in the context of smooth principal bundles over compact manifolds, and as an application thereof in mathematical gauge theory (see, \eg, \cite{Neeb07,Wockel07} and the ref.~therein).
Furthermore, a counterexample to Kaplansky’s famous unit conjecture for group rings has recently been given by Gardam~\cite{Gardam21} by means of Passman’s fours group, which is a group extension of $\mathbb{Z}_2 \times \mathbb{Z}_2$ by~$\mathbb{Z}^3$.

In 1926, Brandt~\cite{Brandt1927} introduced the notion of a groupoid as a generalization of a group.
Since then, the theory of groupoids has flourished into an area of active research and applications of groupoids appear in areas such as
fibre bundle theory,
differential geometry,
foliation theory,
differential topology,
ergodic theory,
functional analysis,
homotopy theory,
and
algebraic geometry,
\cite{Brown1987}.

The classification problem for Abelian groupoid extensions seems to originate in the work of Westman \cite{Westman71}, in which he developed a cohomology theory for groupoids that extends the usual (Abelian) cohomology theory for groups.
About a decade later Renault~reproduced Westman's theory in his pioneering study of C\Star algebras~\cite{Ren80}, thus spotlighting it for operator algebraists and functional analysts.
Another two decades later, Blanco, Bullejos, and Faro~\cite{Blanco2005} studied non-Abelian groupoid extensions from a 2-categorical point of view;
their central result is the classification of non-Abelian groupoid extensions by means of a categorical cohomology theory for groupoids.
Of particular interest is also the article~\cite{CHEN20}, in which the authors study and classify fibrations of Lie groupoids.
In recent years, there has been a renewed interest in groupoid extensions due to the fact that such extensions lead to
many new and interesting algebraic structures (see, \eg,~\cite{ArClCoLiMcRa22,BuMe16,IoKuReiWi21,kumjian_1986,KwLiSk20,Ren80,Ren21} and the ref.~therein).

To illustrate the latter circumstance, let us consider a, possibly non-Abelian, extension of groupoids $\N \to \E \to \G$.
It is natural to ask whether the groupoid ring of $\E$ (resp. the groupoid C\Star algebra of $\E$) can be described in terms of data associated with the building blocks $\N$ and $\G$.
For groups this question has been studied by many authors (see, \eg,~\cite{OiWa22} and the ref.~therein) and leads to the class of \emph{group crossed products}, which is very well-understood and has numerous connections to geometry, operator algebras, and mathematical physics (see, \eg, \cite{AnGuIsRu22,Pass89,SchWa16} and the ref.~therein).
Furthermore, Renault~\cite[Prop.~1.22]{Ren80} proved that the groupoid C\Star algebra of a \emph{twist}, \ie, a groupoid extension by the trivial torus bundle, can be realized as a twisted groupoid C\Star algebra.
A~treatment of the most general case of a proper non-Abelian groupoid extension has, however, to the best of our knowledge not been worked out yet.

Our investigations naturally lead to the class of groupoid crossed products, which is, in contrast to the class of group crossed products, relatively new (cf.~\cite{CaLuPi21,OiLu12}) and thus provides fertile ground for further studies.
More precisely, we establish that each, possibly non-Abelian, groupoid extension $\N \to \E \to \G$ gives rise to a groupoid crossed product of $\G$ by the groupoid ring of $\N$ which recovers the groupoid ring of $\E$ up to isomorphism.
This also provides a natural class of examples of groupoid crossed products.
In addition, it is our hope that this work will contribute to the development and understanding of groupoid C\Star algebras and Steinberg algebras.

Here is an outline of this article.
In Section~\ref{sec:Prelim}, we provide the necessary foundations on groupoids, Abelian groupoid cohomology,
groupoid rings, and groupoid crossed products.
In Section~\ref{sec:extgrp}, we
develop a geometrically oriented
classification theory for non-Abelian extensions of groupoids by means of groupoid cohomology \`a la Westman (see,~\eg,~Corollary~\ref{cor:classext},
Corollary~\ref{cor:simplytransitiv}, and~Corollary~\ref{cor:charclass}).
We~wish to point out that our results, up to Theorem~\ref{teo:equivext}, are similar to the results obtained  by Blanco, Bullejos, and Faro in~\cite{Blanco2005}, but presented in a more geometric and computational framework.
Moreover, from Corollary~\ref{cor:simplytransitiv} and on, our investigation goes further.
In Section~\ref{sec:groupoidrings}, we study groupoid crossed products associated with groupoid extensions (see Theorem~\ref{teo:gpdringiso})
and show 
that the groupoid ring of a groupoid extension is isomorphic to a groupoid crossed product associated with the building blocks of the extension (see Corollary~\ref{cor:gpdringiso}).
We also extend our results to the realm of C\Star algebras (see Proposition~\ref{prop:cstar}).
In Section~\ref{sec:GroupoidCrossedProd}, we make use of the methods developed in Section~\ref{sec:extgrp} to provide a classification theory for groupoid crossed products
(see,~\eg,~Proposition~\ref{prop:equivcrossprod}
and
Theorem~\ref{thm:Classification}).

\section{Preliminaries}\label{sec:Prelim}
In this preliminary section we recall the most fundamental definitions and notation used throughout this article.
Henceforth, all rings are assumed to be associative.

\subsection{Groupoids}\label{sec:groupoids}

There are several ways to view groupoids. 
In this article we consider groupoids as objects with a geometric flavour. 
We refer the reader to \cite{mackenzie,Ren80,sims} for equivalent definitions as well as for examples.

By a \emph{groupoid} we mean a non-empty set $\G$ with a distinguished subset~$\GZ$, called the \emph{unit space} of $\G$, together with structure maps $r,s : \G \to\GZ$, called respectively the \emph{range} and the \emph{source} maps, a partial multiplication $(x,y) \mapsto xy$ in $\G$ defined on the set $\GD:=\{(x,y) \in \G \times \G : s(x) = r(y)\}$ of \emph{composable elements} of $\G$, and a map $\G \ni z \mapsto z^{-1} \in \G$, called \emph{inversion}, satisfying the following properties for all $x,y,z \in \G$ and $u\in \GZ$:
\enlargethispage{\baselineskip}
\begin{itemize}
    \item[(G1)] 
        $r(u) = u = s(u)$;
    \item[(G2)] 
        $r(z)z = z = z s(z)$;
    \item[(G3)] 
        $r(z^{-1}) = s(z)$ and $s(z^{-1}) = r(z)$;
    \item[(G4)] 
        $z^{-1}z = s(z)$ and $zz^{-1} = r(z)$;
    \item[(G5)] 
        $r(xy) = r(x)$ and $s(xy) = s(y)$ whenever $s(x) = r(y)$; 
    \item[(G6)] $(xy)z = x(yz)$ whenever $s(x) = r(y)$ and $s(y) = r(z)$.
\end{itemize}
To emphasize the unit space $\GZ$, we shall sometimes say that $\G$ is a groupoid over~$\GZ$.

Given a groupoid $\G$, we write $\G^{(n)}$ for the set of all $n$-tuples of composable elements of~$\G$, that is, $\G^{(n)}:=\{(x_1,\ldots,x_n)\in \G^{n} : s(x_i)=r(x_{i+1}), \, i=1,\ldots,n-1\}.$
We also bring to mind that a \emph{homomorphism} of groupoids $\G$ and $\HH$ is a map $\phi:\G\to \HH$ such that $\phi(xy)=\phi(x)\phi(y)$ for all $(x,y)\in \GD$ and $\phi(z^{-1})=\phi(z)^{-1}$ for all $z\in \G$. 
Note that each homomorphism $\phi:\G\to \HH$ satisfies $\phi(\GZ)\subseteq \HZ$, thus inducing a map $\phi^0:\GZ\to~\HZ$.
An \emph{isomorphism} of groupoids is simply a bijective homomorphism.

\subsection{Groupoid cohomology}\label{abeliancohomology}

We shall also be concerned with groupoid cohomology.
For convenience of the reader we briefly recall the basics of this theory.
For further reading we refer to~\cite[Sec.~1]{Ren80}.

Let $\C$ be a category and let $X$ be a non-empty set. 
A \textit{$\C$-bundle} over $X$ is a pair $(\N,p)$, where $\N$ is a non-empty set and $p:\N \rightarrow X$ is a map with the property that each fiber $N_u:=p^{-1}({u})$, $u\in X$, is an object of $\C$.
If $\C$ is the category of groups (resp. rings), then $\N$ is called a \emph{group (resp. ring) bundle}. 
In particular, we refer to $\N$ as \textit{Abelian} if each fiber $N_u$ is an Abelian group (resp. commutative ring).
We use the symbol
$\iso_\C(\N)$, or simply $\iso(\N)$, to denote the isomorphism groupoid of the $\C$-bundle $(\N,p)$.

Each group bundle carries a natural groupoid structure.
Indeed, let $X$ be a set and let $(\N,p)$ be a group bundle over $X$. 
For each $u \in X$ denote by $1_u$ the unit of the fiber~$N_u$ and 
put $\N^{(0)} := \{1_u : u\in X\}$.
Define the source and the range of $n\in \N$ to be equal to $1_{p(n)}$. 
Consider the partial multiplication and the inversion defined by the respective operations on the fibers $N_u$, $u \in X$. 
This turns $\N$ into a groupoid over $\N^{(0)}$.
Identifying $\N^{(0)}$ with $X$, in which case $p$ becomes the source and the range map, yields the claim.

Let $\G$ be a groupoid. 
A $\G$-\emph{module bundle} is a pair $((\A,p),L)$, where $(\A,p)$ is an Abelian group (or ring) bundle over $\GZ$ and $L$ is a $\G$-module structure on $\A$, that is, $L$ consists of a family $L_x : A_{s(x)} \rightarrow A_{r(x)}$, $x \in \G$, of group (or ring) isomorphisms such that $L_u=\operatorname{id}_{A_u}$ for all $u \in \GZ$ and $L_xL_y=L_{xy}$ whenever $(x,y)\in \GD.$

Let $((\A,p),L)$ be a $\G$-module bundle.
For $n \in \mathbb{N}_0$ an \emph{n-cochain} is a map $h:\G^{(n)}\rightarrow\A$ satisfying the following conditions:
\begin{enumerate}
    \item 
        $p\left(h(x_1,\ldots,x_n)\right)=r(x_1)$ for every $(x_1,\ldots,x_n) \in \G^{(n)}.$
    \item 
        If $n\geq 1$ and $x_i\in \GZ$ for some $i\in \{1,\ldots,n\}$, then $h(x_1,\ldots,x_n)\in \GZ.$
\end{enumerate}
We denote by $C^n(\G,\A)$ the set of $n$-cochains and define $d^0_L:C^0(\G,\A)\rightarrow C^{1}(\G,\A)$ by 
\begin{align*}
    d^0_L(h)(x):=L_x\left(h(s(x))\right)-h(r(x)).
\end{align*}
For $n > 0$ we consider the map $d^n_L: C^n(\G,\A) \rightarrow C^{n+1}(\G,\A)$ given by
\begin{align*}
    d^n_L(h)(x_1,\ldots,x_{n+1}) &:= L_{x_1}(h(x_2,\ldots,x_{n+1}))
+\sum_{i=1}^{n}(-1)^{n}h(x_1,\ldots,x_ix_{i+1},\ldots,x_{n+1})
    \\
    & \quad +(-1)^{n+1}h(x_1,\ldots,x_n).
\end{align*}
This gives a chain complex $\left(C^n(\G,\A),d^n_L\right)_{n\in \mathbb{N}_0}$.
For $n\in \mathbb{N}_0$ we write $Z^n(\G,\A)_L$ 
for the $n$-\emph{cocycles}, $B^n(\G,\A)_L$
for the $n$-\emph{coboundaries}, and $H^n(\G,\A)_L:=Z^n(\G,\A)_L/ B^n(\G,\A)_L$ for the $n$-th \emph{cohomology group}.

\subsection{Groupoid rings}

Let $\G$ be a groupoid and let $R$ be a unital ring.
We recall that the \textit{groupoid ring} $R[\G]$ is the set of all finitely supported functions $f:\G\rightarrow R$ endowed with the addition given by taking the pointwise sum and the product given by $$(fg)(z) := \sum_{xy=z}f(x)g(y).$$

For a finite subset $F\subseteq\G$, we let $\delta_F\in R[\G]$ stand for the corresponding characteristic function.
In particular, for $F=\{x\}$ we simply write $\delta_x$.

\subsection{Groupoid crossed products}\label{sec:groupoidcrossprod}

In what follows, we recall the foundations on groupoid crossed products (cf.~\cite{CaLuPi21,OiLu12}).

\begin{defi}
Let $\G$ be a groupoid and let $S$ be a ring.
We say that $S$ is \emph{$\G$-graded} if there are additive subsets $S_x$ of $S$, for $x\in \G$, such that
$S = \oplus_{x\in \G} S_x$ and
$S_x S_y \subseteq S_{xy}$ if $(x,y) \in \GD$ and $S_x S_y = \{0\}$ otherwise.
\end{defi}

\begin{defi}
A $\G$-graded ring $S$ is \textit{object unital} if 
for all $u \in \GZ$ the ring $S_u$ is unital, and for all $x \in \G$ and all 
$r \in S_{x}$ the equalities $1_{S_{r(x)}} r = r 1_{S_{s(x)}} = r$ hold.
\end{defi}

\begin{defi}[\cf~{\cite[Def.~10 and Def.~12]{CaLuPi21}}]
Let $\G$ be a groupoid and let $S$ be a $\G$-graded ring which is object unital.
\begin{enumerate}[{\normalfont \rmfamily (i)}]
    \item
        We put $S_0 := \bigoplus_{u \in \GZ} S_u$ and consider $S_0$ as a $\G$-graded ring as follows:
        If $x \in \G$, then $(S_0)_x = S_x$, if $x \in \GZ$, and $(S_0)_x = \{0\}$, otherwise. 
    \item 
        We say that a homogeneous element $r\in S_x$ is \textit{object invertible} if there exists $s\in S_{x^{-1}}$ such that $sr=1_{S_{s(x)}}$ and $rs=1_{S_{r(x)}}$. We denote by $S_{\text{gr}}^\times$ the set of all object invertible elements of $S$.
    \item
        We say that $S$ is a \emph{$\G$-crossed product} if for all $x \in \G$ the relation $S_{\text{gr}}^\times \bigcap S_x \neq \emptyset$ holds. By \cite[Prop.~7(iv)]{CaLuPi21}, all object crossed products are strongly graded,
        \ie, $S_x S_y = S_{xy}$ for all $(x,y)\in \GD$.
        
\end{enumerate}
\end{defi}

\begin{defi}
Let $\G$ be a groupoid, let $\R$ be a unital ring bundle over $\GZ$, and let $R$ be the ring $\bigoplus_{u \in \GZ} R_u$.
\begin{enumerate}[{\normalfont \rmfamily  (i)}]
    \item
        We call a $\G$-crossed product $S$ a \emph{$\G$-crossed product over $R$} if $S_0 =R$.
    \item
        Two $\G$-crossed products $S$ and $S'$ over $R$ are called \emph{equivalent} if there exists a graded isomorphism $\phi: S \to S'$ such that $\phi\lvert_{R} = \id_{R}$.
    \item 
        We let $\Ext(\G,\R)$ stand for the set of all equivalence classes of $\G$-crossed products over $R$.        
        Given a $\G$-crossed product $S$ over $R$, we write $[S]$ for its class in $\Ext(\G,\R)$.
\end{enumerate}
\end{defi}

\begin{defi}[\cf~{\cite[Def.~13]{CaLuPi21}}]\label{def:facsysring}
Let $\G$ be a groupoid, let $\R$ be a unital ring bundle over~$\GZ$, and consider the induced group bundle $\R^\times$ over $\GZ$ given by $\R^\times := \bigcup\limits_{u \in \GZ} R_u^\times$.
\begin{enumerate}[{\normalfont \rmfamily  (i)}]
    \item
        We define $C^1(\G,\iso(\R))$ as the set of all families of maps $\{M_x: R_{s(x)}\rightarrow R_{r(x)}\}_{x\in \G}$ of ring isomorphisms such that $M_u=\operatorname{id}_{R_u}$ for all $u \in \GZ$.
    \item
        We write $C^2(\G,\R^\times)$ for the set of all maps $\tau:\GD\rightarrow \R^\times$ such that $\tau(x,y) \in R_{r(x)}^\times$ for all $(x,y) \in \GD$ and $\tau(x,s(x))=\tau(r(x),x)=r(x)$ for all $x \in \G$.
    \item\label{cond:facsysring}
        We call a pair $(M,\tau) \in C^1(\G,\iso(\R)) \times C^2(\G,\R^\times)$ a \emph{factor system} for $(\G,\R)$ if the following conditions are satisfied:
        \begin{enumerate}
            \item[(C1)]
                $M_xM_y(n) = \tau(x,y)M_{xy}(n)\tau(x,y)^{-1}$ for all $(x,y)\in \GD$ and $n \in R_{s(y)}$,
            \item[(C2)]
            $\tau(x,y)\tau(xy,z)=M_x(\tau(y,z))\tau(x,yz)$ for all $(x,y,z) \in \G^{(3)}$.
        \end{enumerate}
    \item
        We let $Z^2(\G,\R)$ stand for the set of all factor systems for $(\G,\R)$.
\end{enumerate}
\end{defi}

\begin{prop}[\cf~{\cite[Def.~14~and~Prop.~16]{CaLuPi21}}]\label{def:crossedproduct}
Let $\G$ be a groupoid and let $\R$ be a unital ring bundle over~$\GZ$.
For a factor system $(M,\tau)$ for $(\G,\R)$ let $\R \times_{(M,\tau)} \G$ be the set of all functions $f:\G\rightarrow \R$ with finite support satisfying $p\circ f=r$. Then $\R \times_{(M,\tau)} \G$ becomes an associative ring when equipped with the pointwise sum and the product $$\big(fg\big)(z):=\sum_{xy=z}f(x)M_x(g(y))\tau(x,y).$$
Moreover, $\R \times_{(M,\tau)} \G$ is a $\G$-graded ring which is a $\G$-crossed product over $\R$. 
Conversely, any $\G$-crossed product over $\R$ can be presented in this way.
\end{prop}

\begin{rem}\label{rem:identities}
Let $\G$ be a groupoid, let $\R$ be a unital ring bundle over~$\GZ$, and let $(M,\tau)$ be a factor system for $(\G,\R)$.
For all $x,y,z\in \G$ such that $xy=z$ the following identities hold:
\begin{align}
    \tau(x,x^{-1})&=M_x\left(\tau(x^{-1},x)\right),\label{eq:apb1}
    \\
   \tau(z,y^{-1})&=\tau(x,y)^{-1}M_x\left(\tau\left(y,y^{-1}\right)\right),\label{eq:apb2}
    \\
    \tau(z,y^{-1})\tau(x,x^{-1})&=M_z\left(\tau(y^{-1},x^{-1})\right)\tau(z,z^{-1}),\label{eq:apb3}
    \\
    \tau(z,y^{-1})M_x(n)&=M_{z}\left(M_{y^{-1}}(n)\right) \tau(z,y^{-1}), \qquad n\in R_{s(x)}.\label{eq:apb4}
\end{align}
\end{rem}

\section{Non-Abelian extensions of groupoids and their classification}\label{sec:extgrp}

In this section we develop a geometrically oriented classification theory for non-Abelian extensions of groupoids in the spirit of Schreier, Baer, and Eilenberg-Mac Lane.

Throughout the following let $\G$ be a groupoid and let $(\N,p)$ be a group bundle over~$\GZ$, which we shall consider as a groupoid over $\GZ$ with respect to its natural groupoid structure described in Section~\ref{abeliancohomology}.

\begin{defi}
    A \emph{groupoid extension} of $\G$ by $\N$ is a surjective homomorphism \linebreak $j: \E \to \G$, where $\E$ is a groupoid over $\GZ$, $j^{0}$ is the identity map on $\GZ$ and $\N = \ker(j)$, \ie, the set of elements $e \in \E$ such that $j(e) \in \GZ$.
    Usually, we shall write $$\N \to \E \stackrel{j}{\to} \G$$ to denote a groupoid extension of $\G$ by $\N$.

\begin{enumerate}[{\normalfont \rmfamily  (i)}]
    \item
        We call two groupoid extensions $\N \to \E \stackrel{j}{\to} \G$ and $\N \to \E' \stackrel{j'}{\to} \G$ of $\G$ by $\N$ $\emph{equivalent}$ if there exists a groupoid homomorphism 
        $\phi: \E \to \E'$ such that the following diagram commutes:
        \begin{align*}
            \xymatrix{\N \ar[r] \ar[d]^{\id_\N} & \E \ar[r]^j \ar[d]^{\phi}& \G \ar[d]^{\id_\G} \\ \N \ar[r] & \E' \ar[r]^{j'} & \G}
        \end{align*}
        It is easily seen that any such $\phi$ is, in fact, an isomorphism of groupoids.
        We shall sometimes say that such a map $\phi$ is an \emph{equivalence} of groupoid extensions.
    \item
        We denote by $\Ext(\G,\N)$ the set of all equivalence classes of groupoid extensions of $\G$ by $\N$.        
        Given an extension $\E$ of $\G$ by $\N$, we write $[\E]$ for its class in $\Ext(\G,\N)$.
\end{enumerate}
\end{defi}

\begin{rem}
     Note that all groupoids involved in a groupoid extension necessarily have the same unit space.
\end{rem}

\begin{example}\label{ex:fundamentalgroupoid}
    Let $X$ be a path-connected space.
    From homotopy classes of paths in~$X$ we obtain the \emph{fundamental groupoid of $X$}, denoted $\Pi(X)$, 
    which may be considered as a groupoid extension
    \begin{align*}
        \bigcup_{x \in X} \pi_1(X,x) \to \Pi(X) \stackrel{j}{\to} X \times X
    \end{align*}
    of the pair groupoid $X \times X$ by the group bundle $\bigcup_{x \in X} \pi_1(X,x)$, where $j$ is the map that takes a homotopy class of a path $\gamma : [0,1] \to X$ to the pair $(\gamma(1),\gamma(0))$ and $\pi_1(X,x)$ stands for the fundamental group of $X$ with respect to the base point $x \in X$.
\end{example}

\begin{example}\label{ex:grpbundles}
Another geometrically oriented example of a groupoid extension is given as follows:
Let $q: P \to X$ be a locally trivial principal bundle with structure group $G$ and consider the natural action of $G$ on $P \times G$ given by $(p,g).h := (p.h,h^{-1}g)$ for $p \in P$ and $g,h \in G$.
The corresponding quotient $C_G(P):=
(P \times G)/G$ is a group bundle over $X$, the so-called \emph{conjugation bundle}, which is of particular interest in gauge theory, because its space of sections is isomorphic to the gauge group of the principal bundle.
Now, let $N \to E \stackrel{\pi}{\to} G$ be a short exact sequence of, possibly non-Abelian, Lie groups.
Furthermore, suppose that there exists a locally trivial principal bundle $q': P' \to X$ with structure group $E$  such that $P'/N \cong P$.
Then we obtain a short exact sequence of the corresponding conjugation bundles
\begin{align*}
    C_N(P') \to C_E(P') \stackrel{j}{\to} C_G(P) \qquad \text{with} \qquad j([(p',e)]):=[([p'],\pi(e))],
\end{align*}
and therefore an extension of groupoids.
By passing over to the corresponding spaces of sections we get a short exact sequence of gauge groups.
A particular simple example of the above situation is given in case of a trivial principal bundle $q_X: X \times G \to X$, $q_X(x,g) = x$.
We may then look at $q'_X: X \times E \to X$, $q'_X(x,e) = x$, which in turn leads to the following extension of group bundles over $X$:
\begin{align*}
    X \times N \to X \times E \stackrel{j}{\to} X \times G \qquad \text{with} \qquad j(x,e):= (x,\pi(e)).
\end{align*}
\end{example}

\begin{example}\label{ex:twist}
Let $\G$ be a groupoid over $\GZ$.
A \textit{twist} of $\G$ is a groupoid extension of $\G$ by the trivial group bundle $\GZ \times \mathbb{T}$.
Twists and their applications to operator algebras and related fields have recently regained major interest (see, \eg, \cite{kumjian_1986,raeburn1985continuous,Ren80}). 
\end{example}

\begin{example}
    Let $\G$ be a groupoid.
    In \cite[Chap.~1]{mackenzie} the author introduces the notion of a \emph{normal subgroupoid} of $\G$ and of the corresponding quotient groupoid $\G / \N$ of $\G$ by $\N$ with projection map $\text{pr}:\G \to \G / \N$. 
    In particular, each normal subgroupoid $\N$ of $\G$ yields a groupoid extension of the form $\N \to \G \stackrel{\text{pr}}{\to} \G/ \N$.
\end{example}

We proceed to give a description of non-Abelian groupoid extensions in terms of factor systems in analogy with the classical theory of non-Abelian group extensions (see, \eg,~ \cite[Chap.~4]{maclane2012homology}).
 
\begin{defi}\label{def:facsys}
\begin{enumerate}[{\normalfont \rmfamily  (i)}]
    \item
        We define $C^1(\G,\iso(\N))$ to be the set of all families of group isomorphisms $\{L_x: N_{s(x)}\rightarrow N_{r(x)}\}_{x\in \G}$ such that $L_u=\operatorname{id}_{N_u}$ for all $u \in \GZ$.
    \item
        We write $C^2(\G,\N)$ for the set of all maps $\sigma:\GD\rightarrow \N$ such that $\sigma(x,y) \in N_{r(x)}$ for all $(x,y) \in \GD$ and $\sigma(x,s(x))=\sigma(r(x),x)=r(x)$ for all $x \in \G$.
    \item\label{cond:facsys}
        We call a pair $(L,\sigma) \in C^1(\G,\iso(\N)) \times C^2(\G,\N)$ a \emph{factor system} for $(\G,\N)$ if the following conditions are satisfied:
        \begin{enumerate}
            \item[(F1)]
                $L_xL_y(n) = \sigma(x,y)L_{xy}(n)\sigma(x,y)^{-1}$ for all $(x,y)\in \GD$ and $n \in N_{s(y)}$,
            \item[(F2)]
               $\sigma(x,y)\sigma(xy,z)=L_x(\sigma(y,z))\sigma(x,yz)$ for all $(x,y,z) \in \G^{(3)}$.
        \end{enumerate}
        We shall refer to Condition~(\hyperref[cond:facsys]{F1}) as the~\emph{twisted action condition} and to Condition~(\hyperref[cond:facsys]{F2}) as the \emph{twisted cocycle condition}.
    \item
        We let $Z^2(\G,\N)$ stand for the set of all factor systems for $(\G,\N)$.
\end{enumerate}
\end{defi}

\begin{rem}
For fixed $L \in C^1(\G,\iso(\N))$ we denote 
by $Z^2(\G,\N)_L$ the set of all elements $\sigma \in C^2(\G,\N)$ satisfying Condition~(\hyperref[cond:facsys]{F1}) and Condition~(\hyperref[cond:facsys]{F2})  
in Definition~\ref{def:facsys}.
Note that we may then write $Z^2(\G,\N)$ as the disjoint union $$Z^2(\G,\N) = \bigcup_L Z^2(\G,\N)_L,$$ which explains the shift in notation from 2-cocycles $\sigma$ as functions to pairs $(L,\sigma)$.
If $\N$ is a bundle of Abelian groups one can fix $L$ and deal with each set $Z^2(\G,\N)_L$ separately, but that is not possible for bundles of non-Abelian groups.
\end{rem}

The purpose of the following example is to show that every groupoid extension of $\G$ by $\N$ admits a factor system for $(\G,\N)$.
There and subsequently, we use the notation
\begin{align*}
    \N \times_{(p,r)} \G := \{(n,x)\in \N \times \G : p(n)=r(x)\}.
\end{align*}

\begin{example}\label{ex:facsys}
Let $\N \to \E \stackrel{j}{\to} \G$ be a groupoid extension of $\G$ by $\N$.
Furthermore, let $k:\G\rightarrow \E$ be a normalized section for $j$, \ie, $j\circ k = \operatorname{id}_\G$ and $k\lvert_{\GZ} = \operatorname{id}_{\GZ}$.
Then 
\begin{align}
    \phi: \N \times_{(p,r)} \G \to \E, \qquad (n,x) \mapsto nk(x)
\end{align}
is a bijection.
\begin{proof}[Proof of the claim]
For each $x\in \E$ we have $x=xk(j(x))^{-1}k(j(x))$ and $xk(j(x))^{-1}\in \N$, the latter due to the section property. 
This shows that $\phi$ is surjective. 
To establish its injectivity, we assume that $nk(x)=mk(y)$ for some $n,m \in \N$ and $x,y\in \G$. 
Applying $j$ then gives $x=y$, and further $n=m$ by cancellation.
\end{proof}
Now, each $x\in \G$ defines a group isomorphism 
\begin{align*}
    L_x: N_{s(x)} \to N_{r(x)}, \qquad n \mapsto k(x)nk(x)^{-1}.
\end{align*}
Furthermore, the bijectivity of the map $\phi$ implies that $j^{-1}({x})=N_{r(x)}k(x)$ for all $x\in \G$.
Since $j(k(x)k(y))=xy$ for every $(x,y) \in \GD$, we conclude that there exists a unique element $\sigma(x,y)\in N_{r(x)}$ such that 
\begin{equation}\label{kcocycle}
    k(x)k(y)=\sigma(x,y)k(xy).
\end{equation}
This gives a map $\sigma:\GD\rightarrow \N$ with $\sigma(x,y) \in N_{r(x)}$ for all $(x,y) \in \GD$.
These maps are related as follows: for all $(x,y)\in \GD$ and $n\in N_{s(y)}$ we have
\begin{align*}
    L_xL_y(n)
    &= k(x)k(y)n(k(x)k(y))^{-1}
    \\
    &=\sigma(x,y)k(xy)nk(xy)^{-1}\sigma(x,y)^{-1}
    \\
    &=\sigma(x,y)L_{xy}(n)\sigma(x,y)^{-1}.
\end{align*}
Also, associativity entails that $\left(k(x)k(y)\right)k(z)=k(x)\left(k(y)k(z)\right)$ for all $(x,y,z)\in \G^{(3)}$.
The left-hand side is equal to $\sigma(x,y)\sigma(xy,z)k(xyz)$, while
the right-hand side yields
\begin{align*}
    k(x)\left(k(y)k(z)\right)
    =k(x)\sigma(y,z)k(yz)
    =L_x(\sigma(y,z))\sigma(x,yz)k(xyz).
\end{align*}
Consequently, $\sigma(x,y)\sigma(xy,z)=L_x(\sigma(y,z))\sigma(x,yz)$ for all $(x,y,z) \in \G^{(3)}$ by cancellation.
Finally, the fact that the section $k$ is normalized makes it obvious that $(L,\sigma) \in C^1(\G,\iso(\N)) \times C^2(\G,\N)$.
Hence $(L,\sigma)$ is a factor system for $(\G,\N)$.
\end{example}

\begin{rem}
    Let us look again at Example~\ref{ex:fundamentalgroupoid}.
    A (normalized) section of the map $j$, and hence a factor system for $(X \times X, \bigcup_{x \in X} \pi_1(X,x))$, is easily constructed by suitably choosing for each pair $(x,y) \in X \times X$ a path $\gamma: [0,1] \to X$ connecting $x$ and $y$.
\end{rem}

\begin{prop}
Let $(L,\sigma)$ be a factor system for $(\G,\N)$.
Then $\N \times_{(p,r)} \G$ becomes a groupoid over $\{(u,u)\mid u \in \GZ\}\cong\GZ$ equipped with the following structure maps:
\begin{enumerate}[{\normalfont \rmfamily  (i)}]
    \item 
        The source and the range maps are given by $s(n,x):=s(x)$ and $r(n,x):=r(x)$, respectively.
        In particular, two elements $(n,x)$ and $(m,y)$ in $\N \times_{(p,r)} \G$ are composable if and only if $x$ and $y$~are.
    \item
        For $s(n,x)=r(m,y)$ the product is given by $(n,x)(m,y) := (nL_x(m)\sigma(x,y),xy)$.
    \item
        The inversion is given by $(n,x)^{-1} := \left(\sigma\left(x^{-1},x\right)^{-1}L_{x^{-1}}\left(n^{-1}\right),x^{-1}\right)$.
\end{enumerate}
We write $\N \times_{(L,\sigma)} \G$ for the set $\N \times_{(p,r)} \G$ endowed with the above groupoid structure.
\end{prop}
\begin{proof}
(G1)-(G3) and (G5) in Section~\ref{sec:groupoids} are easily checked. 
Here, we just focus on (G4) and (G6). 
Applying the twisted cocycle condition to the triple $(x, x^{-1},x)$ gives $\sigma(x,x^{-1})=L_{x}(\sigma(x^{-1},x)),$ and hence for $(n,x)\in \N \times_{(p,r)} \G$, we have 
\begin{align*}
    (n,x)(n,x)^{-1}
    &=\left(nL_x\left(\sigma(x^{-1},x)^{-1}\right)L_x\left(L_{x^{-1}}(n^{-1})\right)\sigma(x,x^{-1}),r(x)\right)
    \\
    &=\left(nL_x\left(\sigma(x^{-1},x)^{-1}\right) \sigma(x,x^{-1})n^{-1},r(x)\right)=(nn^{-1},r(x))
    =(r(x),r(x)).
\end{align*}
Similarly, we get $(n,x)^{-1}(n,x)=(s(x),s(x)).$ 
Next, let $(n,x),(m,y),\ (l,z)\in \N \times_{(p,r)} \G$ be such that $(x,y,z)\in \G^{(3)}$.
Then a straightforward computation yields
\begin{align*}
    \big((n,x)(m,y)\big)(l,z)
    &=(nL_x(m)\sigma(x,y)L_{xy}(l)\sigma(x,y)^{-1}\sigma(x,y)\sigma(xy,z),xyz)
    \\
    &=(nL_x(m)L_x(L_y(l))L_x(\sigma(y,z))\sigma(x,yz),xyz)
    \\
    &=(nL_x\big(mL_y(l)\sigma(y,z)\big)\sigma(x,yz),xyz)
    = (n,x)\big((m,y)(l,z)\big).
    \qedhere
\end{align*}
\end{proof}

Summarizing, we get the following result:

\begin{cor}
$\N \times_{(L,\sigma)} \G$ is a groupoid extension of $\G$ by $\N$ for any factor system $(L,\sigma)$ for $(\G,\N)$.  
\end{cor}

\begin{prop}\label{prop:equivalence}
Let $\N \to \E \stackrel{j}{\to} \G$ be a groupoid extension of $\G$ by $\N$.
Furthermore, let $k:\G\rightarrow \E$ be a normalized section for $j$, \ie, $j\circ k = \operatorname{id}_\G$ and $k\lvert_{\GZ} = \operatorname{id}_{\GZ}$, and let 
$(L,\sigma)$ be the associated factor system.
Then $\N \times_{(L,\sigma)} \G$ and $\E$ are equivalent groupoid extensions via the map $\phi: \N \times_{(L,\sigma)} \G \to \E$ given by $(n,x) \mapsto nk(x)$.
\end{prop}
\begin{proof}
By Example~\ref{ex:facsys}, it suffices to verify the algebraic conditions.
Indeed, we first note that $\phi(n,s(n))=nk(s(n))=n$ for all $n \in \N$ and $j(\phi(n,x))=j(n)j(k(x))=x$ for all $(n,x) \in \N \times_{(L,\sigma)} \G $.
Now, let $(n,x),(m,y) \in \N \times_{(L,\sigma)} \G $. Then
\begin{gather*}
   \phi((n,x)(m,y))
   =\phi(nL_x(m)\sigma(x,y),xy)=nL_x(m)\sigma(x,y)k(xy)
   \\
   =(nk(x))(mk(y))=\phi(n,x)\phi(m,y).
\end{gather*}
Moreover, since $\phi(u)= u$ for all $u\in \GZ$, we find $\phi((n,x)^{-1})=\phi(n,x)^{-1}$.
\end{proof}

\begin{defi}\label{def:1-cocycle}
    We denote by $C^1(\G,\N)$ the group of all maps $h:\G \to \N$  satisfying $h(x) \in N_{r(x)}$ for all $x \in \G$ and $h(u) = u$
    for all $u \in \GZ$ with respect to the pointwise product.
    Note that this definition extends the definition of 1-cochains in Section~\ref{abeliancohomology} to the non-Abelian case.
\end{defi}

\begin{prop}\label{prop:h.(L,sigma)}
For $h \in C^1(\G,\N)$ and a factor system $(L,\sigma) \in Z^2(\G,\N)$ we define
        \begin{align}
            (h.L)_x(n) &:= h(x) L_x(n) h(x)^{-1}, && x \in \G, n \in N_{s(x)},
            \\
           \label{def:hsigma} (h.\sigma)(x,y) &:= h(x) L_x(h(y)) \sigma(x,y) h(xy)^{-1}, && (x,y) \in \GD.
        \end{align}
    \item
Then $h.(L,\sigma):=(h.L,h.\sigma)$ is a factor system for $(\G,\N)$ and the map $$\alpha:C^1(\G,\N) \times Z^2(\G,\N) \rightarrow Z^2(\G,\N)$$ 
given by $\alpha_h(L,\sigma) := \alpha(h,(L,\sigma)):=h.(L,\sigma)$ defines an action of $C^1(\G,\N)$ on $Z^2(\G,\N)$.
\end{prop}
\begin{proof}
We only show that $h.(L,\sigma)$ satisfies the twisted action condition~(\hyperref[cond:facsys]{F1}) and the twisted cocycle condition~(\hyperref[cond:facsys]{F2}).
Let $(x,y)\in \GD$ and $n \in N_{s(y)}$.
Then
\begin{align*}
   &(h.\sigma)(x,y)(h.L)_{xy}(n)(h.\sigma)(x,y)^{-1}
   \\
   &=h(x)L_x(h(y))\sigma(x,y)L_{xy}(n)\sigma(x,y)^{-1}L_x(h(y))^{-1}h(x)^{-1}
   \\&
   =h(x)L_x(h(y))L_xL_y(n)L_x(h(y))^{-1}h(x)^{-1}
   \\
   &=h(x)L_x\big(h(y)L_y(n)h(y)^{-1}\big)h(x)^{-1}=(h.L)_{x}\big((h.L)_{y}(n)\big),
\end{align*}
which establishes the twisted action condition~(\hyperref[cond:facsys]{F1}).
Now, let $(x,y,z) \in \G^{(3)}$.
Then
\begin{align*}
    &(h.\sigma)(x,y)(h.\sigma)(xy,z)
    \\
    &=h(x)L_x(h(y))L_xL_y(h(z))L_x(\sigma(y,z))\sigma(x,yz)h(xyz)^{-1}
    \\
    &=h(x)L_x\big(h(y)L_y(h(z))\sigma(y,z)h(yz)^{-1}\big)L_x(h(yz))\sigma(x,yz)h(xyz)^{-1}
    \\
    &=(h.L)_x\big((h.\sigma)(y,z)\big)(h.\sigma)(x,yz),
\end{align*}
and the twisted cocycle condition~(\hyperref[cond:facsys]{F2}) is proved.
Next, we show that $\alpha_{h'}\alpha_h=\alpha_{h'h}$ for all $h,h' \in C^1(\G,\N)$. 
For this let $h,h' \in C^1(\G,\N)$, let $(L,\sigma) \in Z^2(\G,\N)$, and let $(x,y)\in \GD$.
We see at once that $h'.(h.L)=(h'h).L$, and hence it remains to verify that $h'.(h.\sigma)=(h'h).\sigma$. 
Indeed,
\begin{align*}
    &h'.(h.\sigma)(x,y)
    =h'(x)(h.L)_x(h'(y))(h.\sigma)(x,y)h'(xy)^{-1}
    \\
    &=h'(x)h(x)L_x(h'(y))L_x(h(y))\sigma(x,y)h(xy)^{-1}h'(xy)^{-1}
    \\
    &=h'h(x)L_x\big((h'h)(y)\big)\sigma(x,y)(h'h)(xy)^{-1}
    =(h'h).\sigma(x,y)
    \qedhere
\end{align*}
\end{proof}

By Proposition~\ref{prop:h.(L,sigma)}, we have an equivalence relation on the set $Z^2(\G,\N)$ of all factor systems given by 
\begin{align*}
    (L,\sigma) \sim (L',\sigma') \qquad \Longleftrightarrow \qquad \left( \exists h \in C^1(\G,\N) \right) \,\,\, (L',\sigma') = h.(L,\sigma).
\end{align*}
That is, two factor systems are equivalent if they are in the same orbit under the action~$\alpha$. 
We denote the corresponding orbit space of~$\alpha$ by $Z^2(\G,\N)/C^1(\G,\N)$.

\begin{teo}\label{prop:equivext}
For two factor systems $(L,\sigma), (L',\sigma') \in Z^2(\G,\N)$ the following conditions are equivalent:
\begin{enumerate}[{\normalfont \rmfamily  (i)}]
    \item
        $\N \times_{(L,\sigma)} \G$ and $\N \times_{(L',\sigma')} \G$ are equivalent groupoid extensions of $\G$ by $\N$.
    \item
       $(L,\sigma) \sim (L',\sigma')$, \ie, there exists $h \in C^1(\G,\N)$ such that $(L',\sigma') = h.(L,\sigma)$.
\end{enumerate}
If these conditions are satisfied, then the map
\begin{align*}
\psi: \N \times_{(L,\sigma)} \G \to \N \times_{(L',\sigma')} \G, \qquad (n,x) \mapsto (nh(x),x)
\end{align*}
is an equivalence of groupoid extensions and, further, all equivalences of extensions $\N \times_{(L,\sigma)} \G \to \N \times_{(L',\sigma')} \G$ are of this form.
\end{teo}
\begin{proof}
Let $\N \times_{(L',\sigma')} \G$ and $\N \times_{(L,\sigma)} \G$ be equivalent groupoid extensions of $\G$ by $\N$ and let $\phi:\N \times_{(L',\sigma')} \G\to \N \times_{(L,\sigma)} \G$ be a homomorphism implementing the equivalence.
Then there exists a map $\phi_0:\N \times_{(L',\sigma')} \G \to \N$ such that $\phi$ has the form $\phi(n,x)=(\phi_0(n,x),x)$.
It is easily seen that $\phi_0(n,x)\in N_{r(x)}$ for all $(n,x) \in \N \times_{(L',\sigma')} \G$ and $\phi_0(n,s(n))=1_{p(n)}$ for all $n \in N$.
Moreover, for each $(n,x) \in \N \times_{(L',\sigma')} \G$ we find
\begin{align*}
\phi(n,x)&=\phi(n,r(x))\phi(r(x),x)=(n,r(x))(\phi_0(r(x),x),x)=(n\phi_0(r(x),x),x).
\end{align*}
Consequently, the map $h:\G\rightarrow \N$ given by  $h(x):=\phi_0(r(x),x)$ belongs to $C^1(\G,\N)$ and satisfies $\phi(n,x)=(nh(x),x)$.
To proceed, let $(n,x),(m,y) \in \N \times_{(L',\sigma')} \G $.
Then $\phi\big((n,x)(m,y)\big)=\phi(n,x)\phi(m,y)$, and hence 
\begin{equation}\label{eq: phi&h0}(nL'_x(m)\sigma'(x,y)h(xy),xy)=(nh(x)L_x(mh(y))\sigma(x,y),xy).
\end{equation} 
Considering $m\in \GZ$ and $y\in \GZ$, we thus get $(L',\sigma')=h.(L,\sigma)$.
If, conversely, $(L',\sigma') = h.(L,\sigma)$ for some $h\in C^1(\G,\N)$, then we define 
$$\phi: \N \times_{(L',\sigma')} \G \to \N \times_{(L,\sigma)} \G, \qquad (n,x) \mapsto (nh(x),x)$$
and the considerations above show that $\phi$ implements an equivalence of groupoids.
\end{proof}

\begin{cor}\label{cor:classext}
The map $Z^2(\G,\N)\rightarrow \Ext(\G,\N)$ sending $(L,\sigma)$ to $[\N \times_{(L,\sigma)} \G]$ induces a bijection $H^2(\G,\N):=Z^2(\G,\N)/C^1(\G,\N) \to \Ext(\G,\N)$.
\end{cor}

In what follows, we call an element $L \in C^1(\G,\iso(\N))$ \emph{outer} if there exists $\sigma \in C^2(\G,\N)$ such that $(L,\sigma)$ satisfies the twisted action condition~(\hyperref[cond:facsys]{F1}).
We emphasize that 
\begin{align*}
   L \sim L' \qquad \Longleftrightarrow \qquad \left( \exists h \in C^1(\G,\N) \right) \,\,\, L' = h.L
\end{align*}
defines an equivalence relation on the set of all outer elements.
Given an outer element $L \in C^1(\G,\iso(\N))$, we denote by $[L]$ the equivalence class of $L$ and call it a $\G$-\emph{kernel}
in accordance with the notion of kernels in the classical theory of non-Abelian extensions of groups (see, \eg,~ \cite[Chap.~4]{maclane2012homology}).

The preceding proposition shows in particular that if $\N \times_{(L,\sigma)} \G$ and $\N \times_{(L',\sigma')} \G$ are equivalent extensions then $[L]=[L']$. 
We write $\Ext(\G,\N)_{[L]}$ for the set of equivalence classes of groupoid extensions of $\G$ by $\N$ corresponding to the $\G$-kernel~$[L]$.
Moreover, we put $Z(\N) := \bigcup_{u\in \GZ} Z(N_u)$
and consider the induced $\G$-module bundle $(Z(\N),L)$ as well as its cohomology theory (\cf~Section \ref{abeliancohomology}).

\begin{teo}\label{teo:equivext}
Suppose that $L \in C^1(\G,\iso(\N))$ is outer with $\Ext(\G,\N)_{[L]} \neq \emptyset$.
Then the following assertions hold:
\begin{enumerate}[{\normalfont \rmfamily  (a)}]
    \item 
        Each class in $\Ext(\G,\N)_{[L]}$ can be represented by one of the form $N\times_{(L,\sigma)}G$.
    \item 
    Let $(L,\sigma')$ and $(L,\sigma)$ be factor systems for $(\G,\N)$. Then $\sigma^{-1}\cdot\sigma'\in Z^2(\G,Z(\N))_{L}$, and moreover $(L,\sigma')\sim (L,\sigma)$ if and only if  $\sigma^{-1}\cdot\sigma'\in B^2(\G,Z(\N))_{L}$.
 \end{enumerate}
\end{teo}
\begin{proof}
\begin{enumerate}[{\normalfont \rmfamily  (a)}]
\item 
    From Proposition~\ref{prop:equivalence} we know that each groupoid extension of $\G$ by $\N$ is equivalent to one of the form $\N \times_{(L',\sigma')} \G$.
    If $[L']=[L]$ and $h\in C^1(\G,\N)$ satisfies $L'=h.L$, then $h^{-1}.(L',\sigma') = (L,h^{-1}.\sigma')$ so that $\sigma'':=h^{-1}.\sigma'$ satisfies $[\N \times_{(L',\sigma')} \G]=[\N \times_{(L,\sigma'')} \G]$, which proves the first claim.
   
\item 
    We first note that $\sigma(x,y)^{-1}\sigma'(x,y)$ is central for every $(x,y)\in \GD$, because $\sigma(x,y)n\sigma(x,y)^{-1}=\sigma'(x,y)n\sigma'(x,y)^{-1}$ for all $(x,y)\in \GD$ and $n\in N_{r(x)}$ by the twisted action condition~(\hyperref[cond:facsys]{F1})
    Now, we check that $\sigma^{-1}\cdot\sigma'$ is a 2-cocycle.
    For this let $(x,y,z) \in \G^{(3)}$.
    Then
    \begin{align*}
        &\big(\sigma^{-1}\cdot\sigma'\big)(x,y)\big(\sigma^{-1}\cdot\sigma'\big)(xy,z)
        \\
        &= \sigma^{-1}(xy,z) \sigma^{-1}(x,y)\sigma'(x,y)\sigma'(xy,z)
        \\
        &=\sigma^{-1}(x,yz)L_x\big( \sigma(y,z)^{-1}\big)L_x\big(\sigma'(y,z)\big)\sigma'(x,yz)
        \\
        &=L_x\big( \sigma(y,z)^{-1}\sigma'(y,z)\big)\sigma^{-1}(x,yz)\sigma'(x,yz)
        \\
        &=L_x\big( \big(\sigma^{-1}\cdot\sigma'\big)(y,z)\big)\big(\sigma^{-1}\cdot\sigma'\big)(x,yz),
    \end{align*}
    where we have used the twisted cocycle condition~(\hyperref[cond:facsys]{F2}) to get the third equation.

    For the second part we first assume that $(L,\sigma')\sim (L,\sigma)$. 
    Then there exists $h\in C^1(\G,\N)$ such that $(L,\sigma')=h.(L,\sigma)$. 
    In particular $L_x(n)=h(x)L_x(n)h(x)^{-1}$ holds for all $x\in \G$ and $n\in N_{s(x)}$, and hence $h\in C^1(\G,Z(\N))$. 
    Since $h$ is central and $\sigma'=h.\sigma$ we further obtain $$\sigma(x,y)^{-1}\sigma'(x,y)=h(x)L_x(h(y))h(xy)^{-1}=d^1_L(h)\in B^2(\G,Z(\N))_{L}.$$ 
    If, conversely, $\sigma^{-1}\sigma'=d^1_L(h)$ for $h\in C^1(\G,Z(\N))$, then $(L,\sigma')=h.(L,\sigma)$.
\qedhere
\end{enumerate}
\end{proof}

\begin{cor}\label{cor:simplytransitiv}
For a $\G$-kernel $[L]$ with $\Ext(\G,\N)_{[L]}\neq\emptyset$ the following map is a well-defined simply transitive action:
\begin{align*}
    H^2(\G,Z(\N))_L \times \Ext(\G,\N)_{[L]}\rightarrow \Ext(\G,\N)_{[L]},\ \ \ \  \left([\rho],[\N\times_{(L,\sigma)} \G]\right) \mapsto [\N \times_{(L,\sigma\cdot \rho)} \G]
\end{align*}
\end{cor}

\begin{example}
    Invoking Example~\ref{ex:fundamentalgroupoid} with $X$ being the figure ``eight'', we easily see that  the fundamental groupoid $\Pi(X)$ is a groupoid extension of the pair groupoid $\G := X \times X$ by the trivial group bundle $\N := X \times \mathbb{F}_2$, where $\mathbb{F}_2$ denotes the free group on two generators.
    Let $L \in C^1(\G,\iso(\N))$ be an outer element for this groupoid extension.
    Since $\mathbb{F}_2$ is centerless, it follows that $H^2(\G,Z(\N))_L$ is trivial, and hence Corollary~\ref{cor:simplytransitiv} implies that $\Ext(\G,\N)_{[L]}$ only contains the class of $\Pi(X)$.
\end{example}

\begin{rem}\label{rem:abelianext}
Suppose $\A$ is an Abelian group bundle. A factor system $(L,\sigma)$ for $(\G,\A)$ consists of a $\G$-module structure $L$ on $\A$, and an element $\sigma \in Z^2(\G,\A)_L$, and we write $\A \times_\sigma \G$ for the corresponding groupoid extension of $\A \times_{(L,\sigma)} \G$.
Furthermore, we have $L \sim L'$ if and only if $L=L'$. 
Hence a $\G$-kernel $[L]$ is the same as a $\G$-module structure $L$ on $\A$ and $\Ext(\G,\A)_L:=\Ext(\G,\A)_{[L]}$ is the set of groupoid extensions of $\G$ by $\A$ for which the associated $\G$-module structure on $\A$ is $L$.
According to Corollary~\ref{cor:simplytransitiv}, the equivalence classes of groupoid extensions correspond to cohomology classes of cocycles, so that the map
\begin{align*}
    H^2(\G,\A)_L \to \Ext(\G,\A)_L, \qquad [\sigma]\mapsto [\A \times_\sigma \G]
\end{align*}
is a well-defined bijection. 
In fact, by \cite[Prop 1.14]{Ren80} it is not only a bijection but also a group isomorphism.
\end{rem}

We conclude this section with a criterion for the nonemptyness of the set $\Ext(\G,\N)_{[L]}$. 
To the best of our knowledge, such a criterion has not been worked out yet.

\begin{lemma}\label{3-cocycle}
    Suppose that $(L,\sigma) \in C^1(\G,\iso(\N)) \times C^2(\G,\N)$ satisfies the twisted action condition~(\hyperref[cond:facsys]{F1}).
    Then the map $\chi_{(L,\sigma)} : \G^{(3)} \to Z(\N)$ given by 
    $$\chi_{(L,\sigma)}(x,y,z) := L_x(\sigma(y,z))\sigma(x,yz) \sigma(xy,z)^{-1} \sigma(x,y)^{-1},\qquad (x,y,z) \in \G^{(3)}$$
    defines a 3-cocycle, \ie, $\chi_{(L,\sigma)} \in Z^{3}(\G,Z(\N))_L$.
\end{lemma}
\begin{proof}
For ease of notation we simply put $\chi := \chi_{(L,\sigma)}$.
Let $(x,y,z)\in \G^{(3)}$, let $m\in N_{r(x)}$, and define $n=L^{-1}_{xyz}(m)$.
Then
\begin{align*}
    &\sigma(x,y)\sigma(xy,z)m\sigma(xy,z)^{-1}\sigma(x,y)^{-1}
    =\sigma(x,y)L_{xy}L_z(n)\sigma(x,y)^{-1} =L_{x}L_yL_z(n)
\end{align*}
and further
\begin{align*}
    &L_{x}L_yL_z(n)
    =L_x\big(\sigma(y,z)L_{yz}(n)\sigma(y,z)^{-1}\big)
    \\
    &=L_x\big(\sigma(y,z)\big)\big(L_xL_{yz}(n)\big)L_x\big(\sigma(y,z)\big)^{-1}    \\
    &=L_x\big(\sigma(y,z)\big)\sigma(x,yz)m\sigma(x,yz)^{-1}L_x\big(\sigma(y,z)\big)^{-1}.
\end{align*} 
Therefore $L_x\big(\sigma(y,z)\big)\sigma(x,yz)$ and $\sigma(x,y)\sigma(xy,z)$ define the same inner automorphism of $N_{r(x)}$ and hence $\chi(x,y,z)=L_x\big(\sigma(y,z)\big)\sigma(x,yz)\sigma(xy,z)^{-1}\sigma(x,y)^{-1}$ is a central element. 
This shows that the map $\chi$ is well-defined.
We proceed to show that $\chi$ lies in the kernel of the map $d^3_L:C^3(\G,Z(\N))_L\rightarrow C^4(\G,Z(\N))_L$ given by $$d^3_L(\chi)(x,y,z,w):=L_x(\chi(y,z,w))\chi(xy,z,w)^{-1}\chi(x,yz,w)\chi(x,y,zw)^{-1}\chi(x,y,z).$$
Below we explicitly write down all the factors that we have to multiply.
We also emphasize that they can be multiplied in any order, because $\chi$ is central. 
\begin{itemize}
    \item 
        $\chi(xy,z,w)^{-1}= \sigma(xy,z)\sigma(xyz,w) \sigma(xy,zw)^{-1}L_{xy}(\sigma(z,w)^{-1}).$ 
    \item  
        $\chi(x,y,zw)^{-1}=\sigma(x,y)\sigma(xy,zw) \sigma(x,yzw)^{-1}L_x(\sigma(y,zw)^{-1}).$
    \item  
        $\chi(x,yz,w)= L_x(\sigma(yz,w))\sigma(x,yzw) \sigma(xyz,w)^{-1} \sigma(x,yz)^{-1}.$ 
    \item  
        $\chi(x,y,z)=L_x(\sigma(y,z))\sigma(x,yz)\sigma(xy,z)^{-1}\sigma(x,y)^{-1}.$  
    \item   
        $L_x(\chi(y,z,w))= L_x\left(L_y\big(\sigma(z,w)\big)\sigma(y,zw)\sigma(yz,w)^{-1}\sigma(y,z)^{-1}\right).$ 
\end{itemize}    
\enlargethispage{\baselineskip}
Moreover, for simplicity of the presentation we introduce the following auxiliary elements:
\begin{itemize}
    \item 
        $n_1:=L^{-1}_{zw}\left(\sigma(z,w)^{-1}L^{-1}_{xy}\big(\sigma(x,y)\big)\right)$,
    \item 
        $n_2:=L^{-1}_{yzw}\left(\sigma(y,zw)^{-1}\sigma(yz,w)\right)$,
    \item
        $n_3:=L_{y}\left(L_z\left(L_w(n_1n_2)L^{-1}_{yz}(\sigma(y,z))\right)\right)$,
    \item
        $n_4:=L_{y}\big(\sigma(z,w)\big)\sigma(y,zw)\sigma(yz,w)^{-1}$.
\end{itemize}
Using repeatedly the twisted action condition~(\hyperref[cond:facsys]{F1}), we obtain
$$\chi(xy,z,w)^{-1}\chi(x,y,zw)^{-1}=\sigma(xy,z)\sigma(xyz,w)L_{xyzw}(n_1)\sigma(x,yzw)^{-1}L_x(\sigma(y,zw))^{-1},$$ 
and further 
$$\chi(xy,z,w)^{-1}\chi(x,y,zw)^{-1}\chi(x,yz,w)=\sigma(xy,z)L_{xyz}(L_w(n_1n_2))\sigma(x,yz)^{-1}.$$ 
It follows that
\begin{equation}\label{eq:4factors}
   \chi(xy,z,w)^{-1}\chi(x,y,zw)^{-1}\chi(x,yz,w)\chi(x,y,z)=\sigma(x,y)^{-1}L_{x}(n_3). 
\end{equation}
To proceed, we look more closely at $n_3$. 
Indeed, since $L_yL_{zw}(n_2)=\sigma(yz,w)\sigma(y,zw)^{-1}$, we conclude that
\begin{align*}
    n_3
    &= L_yL_z\big(L_w(n_1n_2)\big)L_yL_z\big(L^{-1}_{yz}(\sigma(y,z))\big)
    \\
    &=L_yL_z\big(L_w(n_1n_2)\big)\sigma(y,z)
    \\
    &=L_y\left(\sigma(z,w)L_{zw}(n_1)L_{zw}(n_2)\sigma(z,w)^{-1}\right)\sigma(y,z)
    \\
    &=L_y\left(L^{-1}_{xy}(\sigma(x,y))L_{zw}(n_2)\sigma(z,w)^{-1}\right)\sigma(y,z)
    \\
    &=L_y\big(L^{-1}_{xy}(\sigma(x,y))\big)\sigma(yz,w)\sigma(y,zw)^{-1}L_y\big(\sigma(z,w)^{-1}\big)\sigma(y,z).
\end{align*}
Combining the previous expression with Equation~(\ref{eq:4factors}), we get
\begin{gather*}
    \chi(xy,z,w)^{-1} \chi(x,y,zw)^{-1}\chi(x,yz,w)\chi(x,y,z)
    =\sigma(x,y)^{-1}L_{x}(n_3)
    \\
    =L_x\left(\sigma(yz,w)\sigma(y,zw)^{-1}L_y\big(\sigma(z,w)^{-1}\big)\sigma(y,z)\right)
    =L_x\big(n_4^{-1}\sigma(y,z)\big),
\end{gather*} 
and finally that
\begin{gather*}
    d^3_L(\chi)(x,y,z,w)
    =L_x\big(n_4^{-1}\sigma(y,z)\big)L_x(\chi(y,z,w))
    =L_x\big(n_4^{-1}\sigma(y,z)\chi(y,z,w)\big)
    \\
    =L_x\big(n_4^{-1}\chi(y,z,w)\sigma(y,z)\big)
    =L_x\big(n_4^{-1}n_4\sigma(y,z)^{-1}\sigma(y,z)\big)=1_{N_{r(x)}}.
    \qedhere
\end{gather*}
\end{proof}

\begin{teo}\label{thm:charclass}
Suppose that $(L,\sigma), (L',\sigma') \in C^1(\G,\iso(\N)) \times C^2(\G,\N)$ satisfy the twisted action condition~(\hyperref[cond:facsys]{F1}) and that $L'\sim L.$
Then $\chi:=\chi_{(L,\sigma)}$ and $\chi':=\chi_{(L,'\sigma')}$ are cohomologous 3-cocycles in $Z^{3}(\G,Z(\N))_L$.
\end{teo}
\begin{proof}
To begin with, we note that $L'\sim L$ implies that $L'=L$ on the center $Z(\N)$, and hence the cohomology groups $H^3(\G,Z(\N))_L$ and $H^3(\G,Z(\N))_{L'}$ are, in fact, identical.
To show that $\chi$ and $\chi'$ are cohomologous, we first assume that $L'=L$ and recall that in this case $\sigma^{-1}\cdot\sigma'$ takes values in the center by Theorem~\ref{teo:equivext}~(b). 
Since we also have $\sigma^{-1}\cdot\sigma'=\sigma'\cdot\sigma^{-1}$, it follows that
\begin{align*}
    &\chi'(x,y,z)\chi(x,y,z)^{-1}
    \\
    &=L_x\big(\sigma'(y,z)\big)\sigma'(x,yz) \sigma'(xy,z)^{-1}\sigma'(x,y)^{-1}\sigma(x,y)\sigma(xy,z)\sigma(x,yz)^{-1}L_x\big(\sigma(y,z)^{-1}\big) 
    \\
    &= L_x\big(\sigma'(y,z)\big)\big(\sigma'\cdot\sigma^{-1}\big)(x,yz)L_x\big(\sigma(y,z)^{-1}\big) \big(\sigma'^{-1}\cdot\sigma\big)(xy,z)\big(\sigma'^{-1}\cdot\sigma\big)(x,y)
    \\
    &= L_x\big(\big(\sigma'\cdot\sigma^{-1}\big)(y,z)\big)\big(\sigma'\cdot\sigma^{-1}\big)(x,yz) \big(\sigma'^{-1}\cdot\sigma\big)(xy,z)\big(\sigma'^{-1}\cdot\sigma\big)(x,y)
    \\
    &=d^2_L\big(\sigma^{-1}\cdot\sigma'\big)(x,y,z).
    \end{align*}
Now, if $L'=h.L$ for some $h\in C^1(\G,\N)$ and $\theta:=h.\sigma$ is as in Equation~\eqref{def:hsigma}, then Proposition \ref{prop:h.(L,sigma)} implies that $(L',\theta)$ satisfies the twisted action condition, and further
\begingroup
\allowdisplaybreaks
\begin{align*}
    &L'_x(\theta(y,z))\theta(x,yz)h(xyz)
    \\
    &= L'_x(\theta(y,z))h(x)L_x(h(yz))\sigma(x,yz)
    \\
    &=L'_x\big(\theta(y,z)h(yz)\big)h(x)\sigma(x,yz)
    \\
    &=L'_x\big(h(y)L_y(h(z))\sigma(y,z)\big)h(x)\sigma(x,yz)
    \\
    &=L'_x\big(L'_y(h(z))h(y)\big)L'_x(\sigma(y,z))h(x)\sigma(x,yz)
    \\
    &=L'_x\big(L'_y(h(z))h(y)\big)h(x)L_x(\sigma(y,z))\sigma(x,yz)
    \\
    &=L'_x\big(L'_y(h(z))h(y)\big)h(x)\chi(x,y,z)\sigma(x,y)\sigma(xy,z)
    \\~
    &=\chi(x,y,z)L'_x\big(L'_y(h(z))\big)L'_x(h(y))h(x)\sigma(x,y)\sigma(xy,z)
    \\
    &=\chi(x,y,z)L'_x\big(L'_y(h(z))\big)h(x)L_x(h(y))\sigma(x,y)\sigma(xy,z)
    \\
    &=\chi(x,y,z)L'_x\big(L'_y(h(z))\big)\theta(x,y)h(xy)\sigma(xy,z)
    \\
    &=\chi(x,y,z)\theta(x,y)L'_{xy}(h(z))h(xy)\sigma(xy,z)
    \\
    &=\chi(x,y,z)\theta(x,y)h(xy)L_{xy}(h(z))h(xy)\sigma(xy,z)
    \\
    &=\chi(x,y,z)\theta(x,y)\theta(xy,z)h(xyz).
\end{align*}
\endgroup
Hence $\chi=\chi_{(L',\theta)}$, and combining this with the first step completes the proof.
\end{proof}

\begin{cor}\label{cor:indiclass}
Suppose that $L \in C^1(\G,\iso(\N))$ is outer and choose $\sigma \in C^2(\G,\N)$ such that $(L,\sigma)$ satisfies the twisted action condition~(\hyperref[cond:facsys]{F1}).
Then the cohomology class $[\chi_{(L,\sigma)}] \in H^3(\G,Z(\N))_L$ does not depend on the choice of $\sigma$ and is constant on the equivalence class $[L]$.
\end{cor}

On account of Corollary~\ref{cor:indiclass}, each outer element $L \in C^1(\G,\iso(\N))$ gives rise to a characteristic class $\chi(L) \in H^3(\G,Z(\N))_L$.

\begin{cor}\label{cor:charclass}
For a $\G$-kernel $[L]$ we have $\Ext(\G,\N)_{[L]}\neq\emptyset$ if and only if the characteristic class $\chi(L) \in H^3(\G,Z(\N))_L$ is trivial.
\end{cor}
\begin{proof}
If there exists a groupoid extension $\E$ of $\G$ by $\N$ corresponding to $[L]$, then we may \wilog assume that it is of the form $\N \times_{(L,\sigma)} \G$ for some factor system $(L,\sigma)$ for $(\G,\N)$. 
This in particular implies that $\chi_{(L,\sigma)} = 1$ and hence the characteristic class $\chi(L) \in H^3(\G,Z(\N))_L$ is trivial.
If, conversely, $L \in C^1(\G,\iso(\N))$ is outer and $\chi(L) \in H^3(\G,Z(\N))_L$ is trivial, then there exists $\sigma \in C^2(\G,\N)$ and $\rho \in C^2(\G,Z(\N))$ such that $\chi_{(L,\sigma)} = \chi_{(L,\rho^{-1})}$, so that $(L,\sigma \cdot \rho)$ is a factor system.
It follows that $\N \times_{(L,\sigma \cdot \rho)} \G$ is a groupoid extension of $\G$ by $\N$ corresponding to $[L]$.
This completes the proof.
\end{proof}

\section{Groupoid rings of groupoid extensions}\label{sec:groupoidrings}

Here and subsequently, let $\N \to \E \stackrel{j}{\to} \G$ be a groupoid extension.
In this section we associate certain groupoid crossed products (\cf~Section~\ref{sec:groupoidcrossprod}) with the groupoid extension, study their relationship, and establish, as an application, that the groupoid ring of $\E$ is isomorphic to a $\G$-crossed product over the groupoid ring of $\N$.

For a start let us consider the following two equivalence relations on $\GZ$:
We say that two objects $u,v \in \GZ$ are \emph{$\G$-equivalent}, denoted $u \sim_{\G} v$, if there exists $x \in \G$ such that $s_\G(x)=u$ and $r_\G(x)=v$. 
Similarly, we say that two objects $u,v \in \GZ$ are \emph{$\E$-equivalent}, denoted $u \sim_{\E} v$, if there exists $x \in \E$ such that $s_\E(x)=u$ and $r_\E(x)=v$. 
Making use of normalized sections of $j$, we see at once that either $\sim_{\E}$ or $\sim_{\G}$ generate the same partition $\{S_{\lambda}\}_{\lambda \in \Lambda}$ of $\GZ$.

\begin{rem}
For each $\lambda \in \Lambda$ we may look at the set $\E_\lambda := \{x\in \E : s_\E(x),r_\E(x)\in S_\lambda\}$. 
It is easily checked that the family $\{\E_\lambda\}_{\lambda \in \Lambda}$ consists of subgroupoids of $\E$. 
The same procedure applied to $\G$ and $\N$ yields a family $\{\G_\lambda\}_{\lambda \in \Lambda}$ of subgroupoids of $\G$ and a family $\{\N_\lambda\}_{\lambda \in \Lambda}$ of subgroupoids of $\N$, respectively.
We thus get a family 
\begin{align*}
    \N_\lambda \longrightarrow \E_\lambda \stackrel{j\lvert_{\E_\lambda}}{\longrightarrow} \G_\lambda, \quad \lambda \in \Lambda
\end{align*}
of groupoid extensions.
\end{rem}

To proceed, let $\{R_{\lambda}\}_{\lambda \in \Lambda}$ be a family of unital rings.
Below we present two constructions of factor systems in the sense of Definition~\ref{def:facsysring}: 
\begin{enumerate}
    \item\label{eq:ass.fac.sys}
        For each $u\in \GZ$ we put $R_u:=R_{\lambda}[N_u]$, where $\lambda$ is the unique element in $\Lambda$ such that $u \in S_{\lambda}$, and consider the ring bundle $\R:=\bigcup_{u \in \GZ}R_u$ over $\GZ$.
        Furthermore, let $k:\G\rightarrow \E$ be a normalized section for $j$ and let $(L,\sigma)$ be the associated factor system~(\cf~Example~\ref{ex:facsys}).
        Then a straightforward computation shows that the family of maps
        \begin{gather*}
            M_x:R_{s(x)} \rightarrow R_{r(x)}, \qquad M_x(f) := f\circ L_{x}^{-1}, \qquad x \in \G,
            \shortintertext{and}
            \tau:\GD\rightarrow \R^\times, \qquad \tau(x,y) := \delta_{\sigma(x,y)}
        \end{gather*}
        yields a factor system $(M,\tau)$ for $(\G,\R)$.
    \item\label{eq:trivial.fac.sys}
        For each $u\in \GZ$ we put $R'_u:=R_{\lambda}$,  where $\lambda$ is the unique element in $\Lambda$ such that $u \in S_{\lambda}$. 
Then $\R':=\bigcup_{u \in \GZ}R'_u$ is a ring bundle over $\GZ$ and the family of maps
\begin{gather*}
    M'_x = \id:R'_{s(x)} \rightarrow R'_{r(x)},  \qquad x \in \E,
    \shortintertext{and}
    \tau':\E^{(2)}\rightarrow \R^\times, \qquad \tau'(x,y)=1_{R'_s(x)}
\end{gather*}
yields a factor system for $(\E,\R')$.
\end{enumerate}

For expedience we put all of this on record:

\begin{prop}\label{prop:gpdringiso}
Let $\N \to \E \stackrel{j}{\to} \G$ be a groupoid extension and let $\{R_{\lambda}\}_{\lambda \in \Lambda}$ be a family of unital rings, where $\Lambda$ indexes the partition of $\GZ$ \wrt the equivalence relation $\sim_{\G}$.
Then the following assertions hold:
\begin{enumerate}[{\normalfont \rmfamily  (a)}]
    \item\label{prop:assfacsys}
        If $k:\G\rightarrow \E$ is a normalized section for $j$ and $(L,\sigma)$ is the associated factor system,
        then $(M,\tau)$ defined in \ref{eq:ass.fac.sys} above is a factor system for $(\G,\R)$.
    \item
        $(M',\tau')$ defined in \ref{eq:trivial.fac.sys} above is a factor system for $(\E,\R')$.
\end{enumerate}
\end{prop}

\begin{rem}\label{rem:gpdringiso}
\begin{enumerate}[{\normalfont \rmfamily  (a)}]
    \item 
        We shall refer to $(M,\tau)$ as \emph{the factor system associated with $\{R_{\lambda}\}_{\lambda \in \Lambda}$ and $(L,\sigma)$}.
        If $R_\lambda=R$ for all $\lambda\in \Lambda$, then we denote $\R$ by $R[\N]$.
    \item
        We shall refer to $(M',\tau')$ as the \emph{trivial factor system system associated with $\E$ and $\{R_{\lambda}\}_{\lambda \in \Lambda}$}.
        Note that if $R_\lambda=R$ for all $\lambda\in \Lambda$, then the associated groupoid crossed product is simply the groupoid ring $R[\E]$.
\end{enumerate}
\end{rem}

Having disposed these preparatory steps, we are now ready to prove the following:

\pagebreak[3]
\begin{teo}\label{teo:gpdringiso}
Let $\N \to \E \stackrel{j}{\to} \G$ be a groupoid extension and let $\{R_{\lambda}\}_{\lambda \in \Lambda}$ be a family of unital rings, where $\Lambda$ indexes the partition of $\GZ$ \wrt the equivalence relation~$\sim_{\G}$.
Furthermore, let $k:\G\rightarrow \E$ be a normalized section for $j$, let $(L,\sigma)$ be the associated factor system, and let $(M,\tau)$ be the factor system associated with $\{R_{\lambda}\}_{\lambda \in \Lambda}$ and $(L,\sigma)$.
Finally, let $(M',\tau')$ be the trivial factor system associated with $\E$ and $\{R_{\lambda}\}_{\lambda \in \Lambda}$.
Then the respective groupoid crossed products $\R \times_{(M,\tau)}\G$ and $\R'\times_{(M',\tau')}\E$ (\cf~Proposition~\ref{def:crossedproduct}) are isomorphic.
\end{teo}
\begin{proof}
Let us consider the maps
\begin{align*}
    &\Phi:\R' \times_{(M',\tau')}\E\rightarrow \R \times_{(M,\tau)}\G, \qquad \Phi(f)(x)(n):=f(nk(x))
    \shortintertext{and}
    &\Psi:\R \times_{(M,\tau)}\G \rightarrow \R' \times_{(M',\tau')}\E, \qquad \Psi(f)(e):=f(x)(n),
\end{align*}
where in the latter case $(n,x)\in \N\times \G$ is the unique element such that $e=nk(x)$ (cf. Example \ref{ex:facsys}). 
We first establish that $\Phi$ and $\Psi$ are mutual inverses. 
For this, let $f\in \R\rtimes_{(M,\tau)}\G$, $x\in \G$, and $n\in N_{r(x)}$. Then  $\Phi(\Psi(f))(x)(n)=\Psi(f)(nk(x))=f(x)(n).$
Moreover, for $f\in \R'\times_{(M',\tau')}\E$ and $e=nk(x)\in \E$ we have
$ \Psi(\Phi(f))(e)=\Phi(f)(x)(n)=f(nk(x))=f(e)$, which proves the assertion.
Since $\Phi$ is clearly additive, it remains to show that $\Phi$ is multiplicative.
To this end, let $z \in \G$ and let $n' \in N_{r(z)}$.
Then
\begin{equation*}\label{eq:phiiso}
\{(s,t)\in \ED : st=n'k(z) \}
=
\big\{\left(nk(x), L_x^{-1}(m)k(y)\right) : xy=z, \, nm=n'\sigma(x,y)^{-1} \big\}.
\end{equation*}
\begin{proof}[Proof of the equality]
Let $(s,t) \in \ED$ be such that $st=n'k(z)$ and write $s=nk(x)$ and $t=m'k(y)$.
Put $m:=L_x(m')$ and note that $n'k(z)=nk(x)m'k(y)=nm\sigma(x,y)k(xy).$ 
By uniqueness, we may conclude that $xy=z$ and $nm=n'\sigma(x,y)^{-1}$, and therefore $(s,t)=\left(nk(x), L_x^{-1}(m)k(y)\right).$ 
The inverse inclusion follows from multiplication.
\end{proof}
Now, a standard calculation shows that
\begin{align*}
   \Phi(f)\Phi(g)(z)(n')
   &=\ \ \sum_{xy=z}\Big(\Phi(f)(x)M_x(\Phi(g)(y))\tau(x,y)\Big)(n')
   \\
   &=\ \ \sum_{xy=z}\sum_{nm=n'\sigma(x,y)^{-1}}\Phi(f)(x)(n)M_x(\Phi(g)(y))(m)
   \\
   &=\ \ \sum_{xy=z}\sum_{nm=n'\sigma(x,y)^{-1}}f(nk(x))g\left(L_x^{-1}(m)k(y)\right)
   \\
   &=\sum_{st=n'k(z)}f(s)g(t)=\sum_{st=n'k(z)}f(s)M'_s(g(t))\tau'(s,t)
   =\Phi(fg)(z)(n'),
\end{align*}
which in turn completes the proof.
\end{proof}

\begin{cor}\label{cor:gpdringiso}
Under the hypotheses of Theorem~\ref{teo:gpdringiso} with $R_\lambda:=R$ for all $\lambda \in \Lambda$ we have that
$R[\N] \times_{(M,\tau)} \G$ is isomorphic to the groupoid ring $R[\E]$ (\cf~Remark~\ref{rem:gpdringiso}).
\end{cor}

In the remaining part of this section we extend Corollary~\ref{cor:gpdringiso} to the realm of C\Star algebras. 
For this we first need to suitably adapt Definition~\ref{def:facsysring}:



\begin{defi}
Let $\G$ be a groupoid and let $\R$ be a bundle of normed unital \Star algebras over $\GZ$. 
A \emph{\Star factor system for $(\G,\R)$} is a factor system $(M,\tau)$ in the sense of Definition~\ref{def:facsysring} with the additional property that $M$ is a family of isometric \Star isomorphisms and that $\tau(x,y)^{-1}=\tau(x,y)^*$ for all $(x,y)\in \GD$. 
\end{defi}

\begin{example}\label{ex:apb1}
Let $\G$ be a groupoid, let $\N$ be a group bundle over $\GZ$, and let $(L,\sigma)$ be a factor system for $(\G,\N)$.
Then the construction in \ref{eq:ass.fac.sys}, on page \pageref{eq:ass.fac.sys}, with $R_\lambda:=\mathbb{C}$ for all $\lambda\in \Lambda$ yields, in fact, a \Star factor system for $(\G,\mathbb{C}[\N])$.
\end{example}

\begin{prop}
Let $\G$ be a groupoid, let $\R$ be a bundle of normed unital \Star algebras over $\GZ$, and let
$(M,\tau)$ be a \Star factor system for $(\G,\R)$.
Then $\R\times_{(M,\tau)} \G$ becomes a normed \Star algebra when endowed with the norm $\|f\|_{1}:=\sum_{x\in \G}\|f(x)\|$ and the involution $f^*(x):=\tau(x,x^{-1})^{-1}M_x\left(f(x^{-1})\right)^*$, $x \in \G$.
\end{prop}
\begin{proof}
Clearly, $\|\cdot\|_{1}$ is a norm and the involution is linear conjugate and isometric. 
For $f,\ g \in \R\times_{(M,\tau)} \G$ we now check that $\|fg\|_1\leq\|f\|_1\|g\|_1$, $(f^*)^*=f$, and $(fg)^*=g^*f^*$.
Indeed, we have
\pagebreak[3]
\begin{align*}
 \|fg\|_1&\leq
 \sum_{z\in \G}\sum_{\substack{(x,y)\in \GD \\ xy=z}}\|f(x)\|\|g(y)\|
=\sum_{(x,y)\in \GD}\|f(x)\| \|g(y)\|
\\
&\leq 
\Bigg( \sum_{x \in \G}\|f(x)\| \Bigg) 
\Bigg( \sum_{y\in \G} \|g(y)\| \Bigg)= \|f\|_1 \|g\|_1.
\end{align*}
Moreover, for each $x \in \G$ we find
\begin{align*}
    (f^*)^*(x)&=
    \tau(x,x^{-1})^{-1}M_x\left(\tau(x^{-1},x)^{-1}M_{x^{-1}}(f(x))^*\right)^*
    \\
    &=\tau(x,x^{-1})^{-1}M_x\Big(\left(M_{x^{-1}}(f(x))\tau(x^{-1},x)\right)^*\Big)^*
    \\
    &=\tau(x,x^{-1})^{-1}M_xM_{x^{-1}}\left(f(x)\right)M_x\left(\tau(x^{-1},x)\right)
    \\
    &\stackrel{ (\ref{eq:apb1})}{=}\tau(x,x^{-1})^{-1}M_xM_{x^{-1}}\left(f(x)\right)\tau(x,x^{-1})
    =M_{r(x)}(f(x))=f(x).
\end{align*}
Finally, for $z \in \G$ a straightforward computation gives
\begingroup
\allowdisplaybreaks
\begin{align*}
     g^*f^*(z)&=
     \sum_{xy=z}\tau(x,x^{-1})^{-1}M_x\left(g(x^{-1})\right)^*M_x\left(\tau(y,y^{-1})^{-1}M_y\left(f(y^{-1})\right)^*\right)\tau(x,y)
     \\
     &=\sum_{xy=z}\tau(x,x^{-1})^{-1}M_x\left(g(x^{-1})\right)^*M_x\left( M_y\left(f(y^{-1})\right)\tau(y,y^{-1})\right)^*\tau(x,y)
     \\
     &= \sum_{xy=z}\tau(x,x^{-1})^{-1}\Big(\tau(x,y)^{-1}M_x\left(M_y\left(f(y^{-1})\right)\tau(y,y^{-1})\right)M_x\left(g(x^{-1})\right)\Big)^*
     \\
     &=\sum_{xy=z}\tau(x,x^{-1})^{-1}\Big(M_{z}\left(f(y^{-1})\right)\tau(x,y)^{-1}M_x\left(\tau(y,y^{-1})\right)M_x\left(g(x^{-1})\right)\Big)^*
     \\
     &\stackrel{(\ref{eq:apb2})}{=}\sum_{xy=z}\tau(x,x^{-1})^{-1}\Big(M_{z}\left(f(y^{-1})\right)\tau(z,y^{-1})M_x\left(g(x^{-1})\right)\Big)^*
     \\
     &\stackrel{(\ref{eq:apb4})}{=}\sum_{xy=z}\tau(x,x^{-1})^{-1}\Big(M_{z}\left(f(y^{-1})M_{y^{-1}}\left(g(x^{-1})\right)\right)\tau(z,y^{-1})\Big)^*
     \\
     &=\sum_{xy=z}\tau(x,x^{-1})^{-1}\tau(z,y^{-1})^{-1}M_{z}\Big(f(y^{-1})M_{y^{-1}}\left(g(x^{-1})\right)\Big)^*
     \\
     &\stackrel{(\ref{eq:apb3})}{=}\sum_{xy=z}\tau(z,z^{-1})^{-1}M_{z}\Big(f(y^{-1})M_{y^{-1}}\left(g(x^{-1})\right)\tau(y^{-1},x^{-1})\Big)^*
     \\
     &=\tau(z,z^{-1})^{-1}M_{z}\Big(\sum_{xy=z}f(y^{-1})M_{y^{-1}}\left(g(x^{-1})\right)\tau(y^{-1},x^{-1})\Big)^*
     \\
     &=\tau(z,z^{-1})^{-1}M_{z}\left(fg(z^{-1})\right)=(fg)^*(z).
     \qedhere
\end{align*}
\endgroup
\end{proof}

\begin{defi}
Let $\G$ be a groupoid, let $\R$ be a bundle of complex normed unital \Star algebras over $\GZ$, and let $(M,\tau)$ be a \Star factor system for $(\G,\R)$.
The \textit{C\Star algebra for $(\G,\R,M,\tau)$} is the universal enveloping C\Star algebra of the complex normed \Star algebra $(\R\times_{(M,\tau)} \G,\|\cdot\|_{1},^*)$ and will be denoted by $C^*(\G,\R,M,\tau)$.
\end{defi} 

\begin{example}
Let $\N \to \E \stackrel{j}{\to} \G$ be a groupoid extension and let $(M',\tau')$ be the trivial factor system associated with $\E$ and the family $R_\lambda := \mathbb{C}$, $\lambda \in \Lambda$ (\cf \ref{eq:trivial.fac.sys} on page~\pageref{eq:trivial.fac.sys}). 
Then $C^*(\E,\R', M',\tau')$ is the well-known groupoid $C^*$-algebra of $\E$, $C^*(\E)$.
\end{example}

\begin{prop}\label{prop:cstar}
Under the hypotheses of Theorem~\ref{teo:gpdringiso} with $R_\lambda=\mathbb{C}$ for all $\lambda \in \Lambda$ we have that the map $\Phi:\mathbb{C}[\E]\rightarrow \mathbb{C}[\N] \times_{(M,\tau)}\G$ given by $\Phi(f)(x)(n):=f(nk(x))$ is an isometric \Star homomorphism, and therefore the $C^*$-algebras $C^*(\E)$ and $C^*(\G,\mathbb{C}[\N],M,\tau)$ are isomorphic.
\end{prop}

\begin{proof}
We already know from Theorem~\ref{teo:gpdringiso} that $\Phi$ is a ring homomorphism.
Since it is obviously $\mathbb{C}$-linear, we are reduced to proving that $\Phi$ is isometric and $\Phi(f^*)=\Phi(f)^*.$ 
Indeed, a short computation shows that
\begin{align*}
  \|\Phi(f)\|_1&
  =\sum_{x\in \G}\|\Phi(f)(x)\|
  = \sum_{x\in \G}\sum_{n\in \N_{r(x)}}|\Phi(f)(x)(n)|
 \\
  &= \sum_{x\in \G}\sum_{n\in \N_{r(x)}}|f(nk(x)|
  =
  \sum_{z\in \E}|f(z)|=\|f\|.
 \end{align*}
Moreover, for $x\in \G$ and $n\in N_{r(x)}$ we have
\begin{align*}
    \Phi(f)^*(x)(n)&
    =\left(\tau(x,x^{-1})^{-1}M_x\left(\Phi(f)(x^{-1})\right)^*\right)(n)
   \\&
    =M_x\left(\Phi(f)(x^{-1})\right)^*\left(\sigma(x,x^{-1})n\right)
    \\
    &=\left(\Phi(f)(x^{-1})\right)^*\left(L_x^{-1}\left(\sigma(x,x^{-1})n\right)\right)
    \\
    &=\overline{\Phi(f)(x^{-1})\left(L_x^{-1}\left(n^{-1}\sigma(x,x^{-1})^{-1}\right)\right)}
    \\
    &=\overline{f\left(L_x^{-1}\left(n^{-1}\sigma(x,x^{-1})^{-1}\right)k(x^{-1})\right)}
    \\
    &=\overline{f\left(k(x)^{-1}n^{-1}\sigma(x,x^{-1})^{-1}k(x)k(x^{-1})\right)}
    \\
    &=\overline{f\left(k(x)^{-1}n^{-1}\right)}=\Phi(f^*)(x)(n).
    \qedhere
    \end{align*}
\end{proof}

\section{Groupoid crossed products and their classification}\label{sec:GroupoidCrossedProd}

In this section we provide a classification theory for  groupoid crossed products by using the techniques developed in Section~\ref{sec:extgrp}.
The proofs of our statements may be handled in the exact same way as the proofs of the respective statements in Section~\ref{sec:extgrp} and are therefore omitted for the sake of a concise presentation.

Throughout the following let $\G$ be a groupoid.
Furthermore, let $\R$ be a unital ring~bundle over~$\GZ$ and let $\R^\times$
be the induced group bundle over $\GZ$ (\cf~Section~\ref{sec:groupoidcrossprod}).
We start with the following definition:

\begin{defi}[\cf~Definition~{\ref{def:1-cocycle}}]
        We let $C^1(\G,\R^\times)$ stand for the group of all maps $h:\G \to \R^\times$  satisfying
        $h(x) \in R^\times_{r(x)}$ for all $x \in \G$ and $h(u) = 1_{R_{r(u)}}$ for all $u \in \GZ$ with respect to the pointwise product.
\end{defi}

\begin{prop}[\cf~Proposition~{\ref{prop:h.(L,sigma)}}]
For $h \in C^1(\G,\R^\times)$ and a factor system $(M,\tau) \in Z^2(\G,\R)$ we define
        \begin{align}
            (h.M)_x(n) &:= h(x) M_x(n) h(x)^{-1}, && x \in \G, n \in R_{s(x)},
            \\
           \label{def:htau} (h.\tau)(x,y) &:= h(x) M_x(h(y)) \tau(x,y) h(xy)^{-1}, && (x,y) \in \GD.
        \end{align}
    \item
Then $h.(M,\tau):=(h.M,h.\tau)$ is a factor system for $(\G,\R)$ and the map $$\beta:C^1(\G,\R) \times Z^2(\G,\R) \rightarrow Z^2(\G,\R)$$ defines an action of $C^1(\G,\R)$ on $Z^2(\G,\R)$.
\end{prop}

We call two factor systems $(M,\tau)$ and $(M',\tau')$ for $(\G,\R)$ \emph{equivalent}, written with symbols $(M,\tau) \sim (M',\tau')$, if they are in the same orbit under the action $\beta$. 
We denote the corresponding orbit space of $\beta$ by $Z^2(\G,\R)/C^1(\G,\R)$.

\pagebreak[3]
\begin{prop}[\cf~Proposition~{\ref{prop:equivext}}]\label{prop:equivcrossprod}
For two factor systems $(M,\tau), (M',\tau') \in Z^2(\G,\R)$ the following conditions are equivalent:
\begin{enumerate}[{\normalfont \rmfamily  (i)}]
    \item 
        $\R \times_{(M,\tau)} \G$ and $\R \times_{(M',\tau')} \G$ are equivalent.
    \item
       $(M,\tau)\sim (M',\tau')$, \ie, there exists $h \in C^1(\G,\R^\times)$ such that $(M',\tau') = h.(M,\tau)$.
\end{enumerate}
If these conditions are satisfied, then the map
\begin{align*}
\psi: \R \times_{(M,\tau)} \G \to \R \times_{(M',\tau')} \G, \qquad (n,x) \mapsto (nh(x),x)
\end{align*}
is an equivalence of $\G$-crossed products over $\R$ and, further, all equivalences of $\G$-crossed products over $\R$, $\R \times_{(M,\tau)} \G \to \R \times_{(M',\tau')} \G$, are of this form.
\end{prop}

\begin{cor}[\cf~Corollary~{\ref{cor:classext}}]
The map 
$Z^2(\G,\R)\rightarrow \Ext(\G,\R)$ sending $(M,\tau)$ to $[\R \times_{(M,\tau)} \G]$
induces a bijection $H^2(\G,\R):=Z^2(\G,\R)/C^1(\G,\R^{\times}) \to \Ext(\G,\R)$.
\end{cor}    

In accordance with Section~3 we say that an element $M \in C^1(\G,\iso(\R))$ is \emph{outer} if there exists $\tau \in C^2(\G,\R^\times)$ such that $(M,\tau)$ satisfies the twisted action condition~(\hyperref[cond:facsysring]{C1}) and note that
\begin{align*}
   M \sim M' \qquad \Longleftrightarrow \qquad \left( \exists h \in C^1(\G,\R^\times) \right) \,\,\, M' = h.M
\end{align*}
defines an equivalence relation on the set of all outer elements.
Given an outer element $M \in C^1(\G,\iso(\N))$, we write $[M]$ for the equivalence class of $M$ and call it a $\G$-\emph{kernel}.
Proposition~\ref{prop:equivcrossprod} entails that $\R \times_{(M,\tau)} \G \sim \R \times_{(M',\tau')} \G$ implies $[M]=[M']$, \ie, equivalent $\G$-crossed products over $\R$ correspond to the same $\G$-kernel $[M]$.
We denote by $\Ext(\G,\R)_{[M]}$ the set of equivalence classes of $\G$-crossed products over $\R$ corresponding to the $\G$-kernel~$[M]$.
Moreover, we put $Z(\R)^\times := \bigcup_{u\in \GZ} Z(R_u)^\times$
and consider the induced $\G$-module bundle $(Z(\R)^\times,M)$ as well as its cohomology theory (\cf~Section \ref{abeliancohomology}).

\begin{teo}[\cf~Theorem~{\ref{teo:equivext}}]\label{thm:Classification}
Let $M \in C^1(\G,\iso(\R))$ with $\Ext(\G,\R)_{[M]} \neq \emptyset$.
Then the following assertions hold:
\begin{enumerate}[{\normalfont \rmfamily  (a)}]
    \item 
        Each class in $\Ext(\G,\R)_{[M]}$ can be represented by one of the form $\R\times_{(M,\tau)}G$.
    \item
        Let $(M,\tau')$ and $(L,\tau)$ be factor systems for $(\G,\R)$. Then $\tau^{-1}\cdot\tau'\in Z^2(\G,Z(\R)^\times)_{M}$, and moreover $(M,\tau')\sim (M,\tau)$ if and only if  $\tau^{-1}\cdot\tau'\in B^2(\G,Z(\R)^\times)_{M}$.
        \end{enumerate}
\end{teo}

\begin{cor}[\cf~Corollary~{\ref{cor:simplytransitiv}}]
For a $\G$-kernel $[M]$ with $\Ext(\G,\R)_{[M]}\neq\emptyset$ the following map is a well-defined simply transitive action:
\begin{align*}
    H^2(\G,Z(\R)^\times)_M \times \Ext(\G,\R)_{[M]}\rightarrow \Ext(\G,\R)_{[M]},\ \ \ \  \left([\rho],[\R\times_{(M,\tau)} \G]\right) \mapsto [\R \times_{(M,\tau\cdot \rho)} \G].
\end{align*}
\end{cor}

\begin{teo}[\cf~Lemma~\ref{3-cocycle} and Theorem~{\ref{thm:charclass}}]
Suppose that $(M,\tau), (M',\tau')$ {\small $\in C^1(\G,\iso(\R)) \times C^2(\G,\R^\times)$} satisfy the twisted
action condition~(\hyperref[cond:facsysring]{C1}) and that $M'\sim M.$ Then $\xi_{(M,\tau)}$ and $\xi_{(M',\tau')}$ are cohomologous 3-cocycles in $Z^{3}(\G,Z(\R)^\times)_M$, where 
\begin{align*}
    \xi_{(M,\tau)}(x,y,z) := M_x(\tau(y,z))\tau(x,yz) \tau(xy,z)^{-1}, \qquad (x,y,z) \in \G^{(3)}.
\end{align*}
\end{teo}

\begin{cor}[\cf~Corollary~{\ref{cor:indiclass}}]
Suppose that $M \in C^1(\G,\iso(\R))$ is outer and choose $\tau \in C^2(\G,\R^\times)$ such that $(M,\tau)$ satisfies the twisted action condition~(\hyperref[cond:facsysring]{C1}).
Then the cohomology class $[\xi_{(M,\tau)}] \in H^3(\G,Z(\R)^\times)_M$ does not depend on the choice of $\tau$ and is constant on the equivalence class $[M]$.
\end{cor}

On account of Corollary~\ref{cor:indiclass}, each outer element $L \in C^1(\G,\iso(\N))$ gives rise to a characteristic class $\xi(M) \in H^3(\G,Z(\R)^\times)_M$.

\begin{cor}[\cf~Corollary~\ref{cor:charclass}]
For a $\G$-kernel $[M]$ we have $\Ext(\G,\R)_{[M]}\neq\emptyset$ if and only if the characteristic class $\xi(M) \in H^3(\G,Z(\R)^\times)_M$ is trivial.
\end{cor}

\section*{Acknowledgement}
The first named author was financed by the
Coordenação de Aperfeiçoamento de Pessoal de Nível Superior - Brasil (CAPES) - Finance Code 001.

\bibliographystyle{abbrv}
\bibliography{NAEG}

\end{document}